\documentclass[12pt]{amsart}
\usepackage{amscd,amsmath,amsthm,amssymb}
\usepackage[left]{lineno}
\usepackage{color}
\usepackage{stmaryrd}
\usepackage[utf8]{inputenc}
\usepackage{cleveref}
\usepackage{epstopdf}
\usepackage{graphicx}
\definecolor{verylight}{gray}{0.97}
\definecolor{light}{gray}{0.9}
\definecolor{medium}{gray}{0.85}
\definecolor{dark}{gray}{0.6}

 \def\NZQ{\mathbb}               

 \def\ZZ{{\NZQ Z}}
 \def\RR{{\NZQ R}}

 \def\FF{{\NZQ F}}

 \def\frk{\mathfrak}               

 \def\mm{{\frk m}}

 \def\G{{\mathcal G}}

 \def\ab{{\mathbf a}}
 
 \def\xb{{\mathbf x}}

 \def\0b{{\mathbf 0}}

 \def\opn#1#2{\def#1{\operatorname{#2}}} 
 \opn\chara{char} \opn\length{\ell} \opn\pd{pd} \opn\rk{rk}
 \opn\projdim{proj\,dim} \opn\injdim{inj\,dim} \opn\rank{rank}
 \opn\depth{depth} \opn\grade{grade} \opn\height{height}
 \opn\embdim{emb\,dim} \opn\codim{codim}
 
 \opn\Tr{Tr} \opn\bigrank{big\,rank}
 \opn\superheight{superheight}\opn\lcm{lcm}
 \opn\trdeg{tr\,deg}
 \opn\reg{reg} \opn\lreg{lreg} \opn\ini{in} \opn\lpd{lpd}
 \opn\size{size} \opn\sdepth{sdepth}
 \opn\link{link}\opn\fdepth{fdepth}\opn\lex{lex}
 \opn\tr{tr}
 \opn\type{type}
 \opn\gap{gap}
 \opn\arithdeg{arith-deg}
 \opn\HS{HS}
 \opn\GL{GL}
 \opn\div{div} \opn\Div{Div} \opn\cl{cl} \opn\Cl{Cl}
 \opn\Spec{Spec} \opn\Supp{Supp} \opn\supp{supp} \opn\Sing{Sing}
 \opn\Ass{Ass} \opn\Min{Min}\opn\Mon{Mon}
 \opn\Ann{Ann} \opn\Rad{Rad} \opn\Soc{Soc}\opn\Deg{Deg}
 \opn\Im{Im} \opn\Ker{Ker} \opn\Coker{Coker} \opn\Am{Am}
 \opn\Hom{Hom} \opn\Tor{Tor} \opn\Ext{Ext} \opn\End{End}
 \opn\Aut{Aut} \opn\id{id}
 
 \opn\nat{nat}
 \opn\pff{pf}
 \opn\Pf{Pf} \opn\GL{GL} \opn\SL{SL} \opn\mod{mod} \opn\ord{ord}
 \opn\Gin{Gin} \opn\Hilb{Hilb}\opn\sort{sort}
 \opn\PF{PF}\opn\Ap{Ap}
 \opn\mult{mult}
 \opn\bight{bight}
 \opn\aff{aff}
 \opn\relint{relint} \opn\st{st}
 \opn\lk{lk} \opn\cn{cn} \opn\core{core} \opn\vol{vol}  \opn\inp{inp} \opn\nilpot{nilpot}
 \opn\link{link} \opn\star{star}\opn\lex{lex}\opn\set{set}
 \opn\width{wd}
 \opn\Fr{F}
 \opn\QF{QF}
 \opn\G{G}
 \opn\type{type}\opn\res{res}
 \opn\conv{conv}
 \opn\Ind{Ind}
 \opn\gr{gr}
 
 \def\pot#1#2{#1[\kern-0.28ex[#2]\kern-0.28ex]}

 \opn\dirlim{\underrightarrow{\lim}}
 \opn\inivlim{\underleftarrow{\lim}}
 \let\union=\cup

 \let\iso=\cong
 \let\Union=\bigcup
 
 \let\Dirsum=\bigoplus
 
 \let\to=\rightarrow
 
 \def\Implies{\ifmmode\Longrightarrow \else
         \unskip${}\Longrightarrow{}$\ignorespaces\fi}
 \def\implies{\ifmmode\Rightarrow \else
         \unskip${}\Rightarrow{}$\ignorespaces\fi}
 \def\iff{\ifmmode\Longleftrightarrow \else
         \unskip${}\Longleftrightarrow{}$\ignorespaces\fi}

 \let\:=\colon
 \newtheorem{Theorem}{Theorem}[section]
 \newtheorem{Lemma}[Theorem]{Lemma}
 \newtheorem{Corollary}[Theorem]{Corollary}
 \newtheorem{Proposition}[Theorem]{Proposition}

 \newtheorem{Definition}[Theorem]{Definition}

 \let\epsilon\varepsilon
 \let\kappa=\varkappa
 \def\qed{\ifhmode\textqed\fi
       \ifmmode\ifinner\quad\qedsymbol\else\dispqed\fi\fi}
 \def\textqed{\unskip\nobreak\penalty50
        \hskip2em\hbox{}\nobreak\hfil\qedsymbol
        \parfillskip=0pt \finalhyphendemerits=0}
 \def\dispqed{\rlap{\qquad\qedsymbol}}
 
 \opn\dis{dis}
 \def\pnt{{\raise0.5mm\hbox{\large\bf.}}}
 
 \opn\Lex{Lex}

\begin{document}

\title{Some homological properties of Borel type ideals}

\author{ J\"urgen Herzog, Somayeh Moradi$^{\ast}$, Masoomeh Rahimbeigi and Guangjun Zhu     }

\address{J\"urgen Herzog, Fakult\"at f\"ur Mathematik, Universit\"at Duisburg-Essen, 45117
Essen, Germany} \email{juergen.herzog@uni-essen.de}

\address{Somayeh Moradi, Department of Mathematics, School of Science, Ilam University,
P.O.Box 69315-516, Ilam, Iran}
\email{so.moradi@ilam.ac.ir}

\address{Masoomeh Rahimbeigi, Department of Mathematics, University of Kurdistan, Post
Code 66177-15175, Sanandaj, Iran}
\email{rahimbeigi$_{-}$masoome@yahoo.com}

\address{Guangjun Zhu, School of Mathematical Sciences, Soochow University, Suzhou
215006, P. R. China}
\email{zhuguangjun@suda.edu.cn}

\thanks{$^{\ast}$ Corresponding author}



\maketitle
\begin{abstract}
We study ideals of Borel type, including $k$-Borel ideals and $t$-spread Veronese ideals.  We determine their free resolutions and their homological shift ideals.  The multiplicity and the analytic spread of equigenerated  squarefree principal Borel ideals are computed. For the multiplicity, the result is given under an additional assumption  which is always satisfied for squarefree principal Borel ideals. These results are used to analyze the behaviour of height, multiplicity and analytic spread  of the homological shift ideals $\HS_j(I)$ as functions of $j$, when $I$ is an equigenerated squarefree Borel ideal. 
\end{abstract}
\section*{Introduction}
Let $K$ be a field,  $S=K[x_1,\ldots,x_n]$ the polynomial ring in $n$ variables over $K$ and $I\subset S$ a graded ideal.  
By a famous theorems   of Galligo~\cite{G} and Bayer-Stillman~\cite{BS}, the generic initial ideal of $I$ is Borel-fixed, that is, it is fixed under the action of the Borel subgroups of $GL(n,K)$. Moreover, if  $\chara(K)=0$, then the Borel-fixed ideals are precisely the strongly stable ideals (also known as Borel ideals),  see \cite[Proposition 4.2.4]{HH}. Applying the Kalai stretching operator (\cite{Ka1} and \cite{Ka2}) to strongly stable ideals, one obtains  squarefree strongly stable ideals,  which were first considered in \cite{AHH} and which play an important role in algebraic shifting theory.

In this paper, we focus our attention on $k$-Borel ideals. 
Let $k$ be a positive integer. We call a  monomial $u= x_1^{a_1}\cdots x_n^{a_n}$ {\em $k$-bounded},  if $a_i\leq k$ for $i=1,\ldots,n$. A monomial ideal  which is generated by $k$-bounded monomials   is said to be  {\em $k$-Borel}, if  for any $k$-bounded monomial $u\in I$, and for any $j\in \supp(u)$ and $i<j$, we have $x_i(u/x_j)\in I$,  provided   $x_i(u/x_j)$ is again  $k$-bounded. Thus $k$-Borel ideals are strongly stable ideals respecting the $k$-boundedness. The squarefree strongly stable  ideals are just the $1$-Borel ideals.   Other interesting  restrictions of strongly stable ideals have been considered in \cite{CKST} and \cite{DHQ}.

Eliahou and Kervaire~\cite{EK} gave an explicit free resolution not only for strongly stable ideals, but also for the larger class of stable ideals. This important result allowed Bigatti \cite{Bi}  and Hulett \cite{Hu} to show that if $I$ is a graded ideal, then the graded Betti numbers of $I$ are bounded above by  the graded Betti numbers of the corresponding lex-ideal. 

In Section~1, we state and prove some basic properties of $k$-Borel ideals. For any  finite set $\{u_1,\ldots,u_m\}$  of $k$-bounded monomials, there exists a unique smallest $k$-Borel ideal  containing $u_1,\ldots,u_m$,    which we denote by $B_k(u_1,\ldots,u_m)$. For $k$-bounded monomials $u,v$ of the same degree, we set $v\preceq_{k} u$  if $v\in B_k(u)$. This binary relation defines a partial order on the set of $k$-bounded monomials of the same  degree. In Proposition~\ref{height} we show that the height of  $B_k(u_1,\ldots,u_m)$ is given by $\max\{\min(u_1),\ldots,\min(u_m)\}$, where we set $\min(u)=\min\{i\: x_i|u\}$ for a monomial $u$.

The graded Betti numbers of a $k$-Borel ideal $I$  are computed in Section~2,   see Corollary~\ref{somayeh}.  This is achieved by  Theorem~\ref{Koszulcycles} which provides   a natural $K$-basis of the Koszul homology $H(x_1,\ldots,x_n;S/I)$. Interestingly,  the basis consists of the homology classes of monomial Koszul cycles. The proof of Theorem ~\ref{Koszulcycles} also yields a $K$-basis of $H(x_1,\ldots,x_i;S/I)$ for any initial sequence $x_1,\ldots, x_i$ with $1\leq i\leq n$. By using Corollary~\ref{somayeh} it is shown in Corollary~\ref{anyorder} that if $I$ is a $k$-Borel ideal generated in a single degree, then $I$ has linear quotients for any order of the  monomial generators of $I$,  which extends the  partial order $\prec_k$. This result is complemented by Proposition~\ref{lq} in which we show that any $k$-Borel ideal (even if it is not generated in a single degree) has linear quotients with respect to the lexicographical order of its generators. The explicit resolution of a $k$-Borel ideal $I$ can be given, due to the fact that $I$ has a regular decomposition function, as shown in Proposition~\ref{regdec}. This proposition together with \cite[Theorem 3.10]{DM} also implies that the resolution of $I$ is cellular and supported on a regular CW-complex. 

In Section~3, we determine the multiplicity and the analytic spread of squarefree Borel ideals, which by definition are just the $1$-Borel ideals. The results refer to the block decomposition of the support of a squarefree monomial. This concept was   first introduced in \cite{DHQ}. Let $u= x_{i_1} x_{i_2}\cdots x_{i_d} $ be a squarefree monomial with $i_1<i_2<\cdots <i_d$. A {\em block} of $u$ is a subset $\{i_{\l}, i_{\l+1},\cdots, i_{\l+k}\}$ such that $i_{\l+j}=i_{\l}+j$ for $j=1,\ldots, k$.   A block of $u$  is called {\em maximal} if it is not properly contained in any other block of $u$. Note that $\supp(u)$ can be uniquely decomposed into  maximal blocks. In other words, $\supp(u)=B_{1}\sqcup B_{2}\sqcup\cdots \sqcup B_{k}$, where each $B_i$ is a maximal block and $\max\{j:\, j\in B_i\}<\min\{j:\, j\in B_{i+1}\}-1$ for all $i$. It is shown  in Theorem~\ref{multi} that if $I$ is a squarefree principal Borel ideal with the Borel generator $u$,  then the multiplicity   of $S/I$ is given by  ${\max(B_1) \choose |B_1|-1}$, where $B_1$ is the first block in the block decomposition of $u$. This result can be generalized to squarefree Borel ideals with several Borel generators, provided there is one of the Borel generators  whose first block is contained in the first blocks of all the other Borel generators, see Theorem~\ref{multigeneral}. The analytic spread of a squarefree Borel ideal $I$  with Borel generators $u_1,\ldots,u_m$  is also  determined by the first blocks of the block decomposition, as shown in Theorem~\ref{analytic spread}. If $B_{i1}$ is the first block in the block decomposition of $u_i$, then the analytic spread of $I$ is  $n=\max\{\max(u_i)\;\:  1 \leq i\leq m\}$, if $1\notin \bigcap_{i=1}^m B_{i1}$, and is equal to $n-|\bigcap_{i=1}^m B_{i1}|$, otherwise.

 Section 4 is devoted to the study of the homological shift ideals of equigenerated squarefree  Borel ideals. For the moment, let $I$ be any monomial ideal and let $\FF$ be its minimal  multigraded free $S$-resolution. Then $F_j=\Dirsum_{k=1}^{b_j} S(-{\ab_{jk}})$ with   each $\ab_{jk}$  an integer vector in $\mathbb{Z}^n$ with non-negative entries. The monomial ideal $\HS_j(I)$ generated by the monomials $\xb^{\ab_{jk}}$, $k=1,\ldots,b_j$ is called the $j$th homological shift ideal of $I$.  These ideals provide some extra information about the nature of the multigraded free resolution of $I$. Homological shift ideals have been studied in \cite{Ba},\cite{BJT} and \cite{HMRZ}. It is conjectured that if $I$ has linear resolution, then  $\HS_j(I)$ has linear resolution for all $j$, This conjecture is still widely open, and only proved in some special cases. Here we study how the height, the analytic spread and the multiplicity of $\HS_j(I)$ behave as a function  of $j$, when $I$ is an equigenerated squarefree  Borel ideal. Based on Proposition~\ref{gen1}, we describe in Corollary~\ref{clever} the minimal set  of monomial generators of $\HS_j(I)$. From this description it can be seen that $\HS_j(I)$ is again a squarefree Borel ideal and that  $\HS_1(\HS_j(I))=\HS_{j+1}(I)$ for all $j<\projdim S/I$. Having these information it is not so hard to see that the height and the analytic spread of $\HS_j(I)$ is a non-decreasing function of $j$, see Corollary~\ref{nond}. The multiplicity function of homological shift ideals behaves differently. Indeed, in Proposition~\ref{easy but lengthy} we show that if $I$ is a squarefree  principal Borel ideal, the multiplicity of $\HS_j(I)$ is a unimodal function of $j$. 

 In the last  section of this paper,  Section~5,  we consider the homological shift ideals  of  $t$-spread  Veronese ideals $I_{n,d,t}$ of degree $d$ in $n$ variables. These ideals naturally generalize squarefree Veronese ideals, and have been studied  in several papers (see for example \cite{EHQ} and \cite{AEL}). In particular, it is known that $I_{n,d,t}$ has linear quotients.  A monomial   $x_{i_1}x_{i_2}\cdots x_{i_d}$ with $i_1\leq i_2\leq \dots \leq i_d$ is called {\it $t$-spread} if $i_j -i_{j-1}\geq t$ for $2\leq j\leq n$, and $I_{n,d,t}$ is the ideal  generated by all $t$-spread monomials of degree $d$ in $K[x_1,\ldots,x_n]$. Our Theorem~\ref{gent} describes the minimial set of monomial generators of $\HS_i(I_{n,d,t})$ for each $i$.  This result is used to show  that $\HS_1(I_{n,d,t})$ has linear quotients, see Theorem~\ref{hs1}. We expect that  is true also for the higher homological shift ideals. Actually, one could expect that for any equigenerated monomial ideal with linear quotients, all homological shift ideals have linear quotients.

\section{$k$-Borel ideals}

Throughout this paper, $S=K[x_1,\ldots,x_n]$ is a polynomial ring over a field $K$, $\Mon(S)$ denotes the set of monomials in $S$ and for a monomial $u$, we set $\supp(u)=\{i:\ x_i|u\}$. For a monomial ideal $I$, the unique minimal set of monomial generators of $I$ is denoted by $G(I)$. A monomial ideal $I$ is called {\em $k$-bounded}, if it is generated by $k$-bounded monomials.

In this section we define $k$-Borel ideals as a generalization of squarefree strongly stable ideals and study their basic properties.

\begin{Definition}
{\em Let $k$ be a positive integer and $I$ be a $k$-bounded monomial ideal. We say that $I$ is {\em $k$-Borel}, if for any $k$-bounded monomials $u\in I$ the following condition holds: if $j\in \supp(u)$ and $i<j$, then  $x_i(u/x_j)\in I$ provided that $x_i(u/x_j)$ is $k$-bounded.}
\end{Definition}

\begin{Lemma}
Let $I$ be a $k$-Borel ideal. The following conditions are equivalent:
\begin{itemize}
\item[(i)] $I$ is a $k$-Borel ideal,
\item[(ii)] For any $u\in G(I)$, any $j\in \supp(u)$ and $i<j$, if $x_i(u/x_j)$ is $k$-bounded, then $x_i(u/x_j)\in I$.
\end{itemize}
\end{Lemma}

\begin{proof}
$(\textrm{i})\Rightarrow (\textrm{ii})$ is clear.

$(\textrm{ii})\Rightarrow (\textrm{i})$: Let  $u\in I$ be a $k$-bounded monomial, $j\in \supp(u)$ and $i<j$ such that $x_i(u/x_j)$ is $k$-bounded. Then there exists $w\in G(I)$ such that $w\mid u$. If $x_j\nmid w$, then $w\mid x_i(u/x_j)$ and $x_i(u/x_j)\in I$. If $x_j\mid w$, then $x_i(w/x_j)$ divides $x_i(u/x_j)$. Hence $x_i(w/x_j)$ is $k$-bounded as well. So by our assumption it belongs to $I$ which implies that $x_i(u/x_j)\in I$.
\end{proof}

It is clear that the intersection of $k$-Borel ideals is $k$-Borel as well. Hence for $k$-bounded monomials $u_1,\ldots,u_m$,
there exists a unique smallest $k$-Borel ideal containing $u_1,\ldots,u_m$. Indeed, the set of $k$-Borel ideals containing $u_1,\ldots,u_m$ is not empty, because the maximal ideal of $S$ belongs to this set. The intersection of all $k$-Borel ideals containing $u_1,\ldots,u_m$ is the $k$-Borel ideal we are looking for. We denote this ideal by $B_k(u_1,\ldots,u_m)$ and we call the monomials $u_1,\ldots,u_m$ the {\em Borel generators} of this ideal. A $k$-Borel ideal with one Borel generator is called {\em principal $k$-Borel}.  

We set  $v\preceq_{k} u$, if  $v$ and $u$ are  $k$-bounded monomials of the same degree and  $v\in B_k(u)$. Note that $\preceq_{k}$ is a partial order 
on the set of $k$-bounded monomials of the same degree.  We have  $$B_k(u_1,\ldots,u_m)=(v\in \Mon(S):\ v\preceq_{k} u_i \ \textrm{for some} \ 1\leq i\leq m).$$ In particular, it follows that $B_k(u_1,\ldots,u_m)=\sum_{i=1}^mB_k(u_i)$.

Let $I$ be a strongly stable ideal. In the sequel we will call it a Borel ideal.  Then for $k\gg 0$, $I$ is $k$-Borel, since large  $k$ imposes no conditions on the exponents of the generators. It follows that  for any monomials $u_1,\ldots,u_m$, there is a  unique smallest Borel  ideal, denoted $B(u_1,\ldots,u_m)$, which contains   $u_1,\ldots,u_m$. 

\medskip
For the monomials 
$u=x_{i_1}\cdots x_{i_d}$  and $v=x_{j_1}\cdots x_{j_d}$ with $i_1\leq \cdots \leq i_d$ and $j_1\leq \cdots \leq j_d$, we set  $v\preceq u$ if $j_{\ell}\leq i_{\ell}$ for any $1 \leq \ell \leq d$.
The following easy but useful lemma explains the generators of a principal Borel ideal. 

\begin{Lemma} \label{lemorder}
Let  $u$  and $v$ be monomials of the same degree. Then $v\in B(u)$ if and only if $v\preceq u$. 
\end{Lemma}

\begin{proof}
Let $u=x_{i_1}\cdots x_{i_d}$, $v=x_{j_1}\cdots x_{j_d}$ and  $v\preceq u$. Then the number $\delta(v,u)= \sum_{k=1}^d (i_k-j_k)$ is non-negative and $\delta(v,u)=0$ if and only of $u=v$. We proceed by induction on $\delta(v,u)$ to show that $v\in B(u)$. If $\delta(v,u)=0$, then the assertion is trivial. Assume now that $\delta(v,u)>0$. Then  $i_k-j_k>0$ for some $k$. Let $k$ be the smallest such integer. Then $i_{k-1}=j_{k-1}<j_k<i_k$ and it follows that $u'=x_{i_k-1}(u/x_{i_k})\in B(u)$. Furthermore,  $v\preceq u'\prec u$ and $\delta(v,u')<\delta(v,u)$. Our induction hypothesis implies that $v\in B(u')\subseteq B(u)$.
Conversely, let $v\in B(u)$. It can be easily seen that the ideal $I=(w\in\Mon(S):\ w\preceq u)$ is a Borel ideal containing $u$. This implies that $B(u)\subseteq I$ and hence $v\preceq u$.
\end{proof}

\medskip
For a monomial ideal $I$, we set $I^{\leq k}=(u\in G(I): \ u\  \textrm{is $k$-bounded})$ and call it the  {\em $k$-bounded part of $I$}.
The next result shows that a $k$-bounded monomial ideal $I$ is a $k$-Borel ideal if and only if it is the $k$-bounded part of a Borel ideal. 

\begin{Lemma}
\label{boundedpart}
Let $u_1,\ldots,u_m$ be $k$-bounded monomials. Then $$B_k(u_1,\ldots,u_m)= B(u_1,\ldots,u_m)^{\leq k}.$$ 
\end{Lemma}

\begin{proof}
Let $u$ be a $k$-bounded monomial. By  \cite[Lemma 2.4]{HLR}, $B(u)^{\leq k}=B_k(u)$. Therefore, since $B(u_1,\ldots,u_m)=\sum_{i=1}^mB(u_i)$,  it follows that 
\[
B(u_1,\ldots,u_m)^{\leq k}=(\sum_{i=1}^mB(u_i))^{\leq k}=\sum_{i=1}^mB(u_i)^{\leq k}=\sum_{i=1}^mB_k(u_i)=B_k(u_1,\ldots,u_m). 
\]
\end{proof}

We use the following fact repeatedly in the later proofs.

\begin{Corollary}
 \label{partial}
Let $k$ be a positive integer and $u$ and $v$ be $k$-bounded monomials of the same degree. Then 
 $v \preceq_{k}u$ if and only if $v\preceq u$.
 \end{Corollary}

\begin{proof}
We have  $v \preceq_{k}u$ if and only if $v\in B_k(u)$. Also by ~\Cref{lemorder}, $v\preceq u$ if and only if $v\in B(u)$. Then the desired results  follow from Lemma~\ref{boundedpart} (or even from  \cite[Lemma 2.4]{HLR}). 
\end{proof}

For a monomial $u$, we set 
\[
\min(u)=\min \{i\: i\in \supp{(u)}\}\ \ \text{and \ } \max(u)=\max \{i\: i\in \supp{(u)}\}.
\]
The next result describes the height of a $k$-Borel ideal in terms of its Borel generators.

\begin{Proposition}
\label{height}
Let $I=B_k(u_1,\ldots, u_m)$. Then 
\[
\height{(I)}=\max\{\min(u_1),\ldots, \min(u_m)\}.
\]
\end{Proposition}
\begin{proof}
First we consider the case $m=1$, and hence we may assume that $I=B_k(u)$ where  $u=x_{i_1}^{a_1} x_{i_2}^{a_2}\cdots x_{i_d}^{a_d}$ with $i_1<i_2<\cdots<i_d$ and $a_1>0$. We show that $\height(I)=i_1$. Since $I$ is a monomial ideal, all minimal prime ideals of $I$ are generated by variables. Let $P=(x_1,x_2,\ldots, x_{i_1})$, and let $v\in G(I)$. By~\Cref{partial}, $v\preceq u$. It follows that $1\leq \min(v) \leq i_1$ by the definition of the partial order $\preceq$. Therefore  $v\in P$ for all $v\in I$. This shows that $I\subseteq P$ and proves that $\height(I)\leq i_1$.

Suppose now $\height(I)<i_1$. Then there exists a monomial prime ideal $Q$ containing $I$ with less than $i_1$ generators. We show that this is not possible. Indeed, we show that there exist positive integers $j_1< j_2<\cdots< j_d$ with $j_k\leq i_k$ for all $k$ such that $x_{j_k}\notin Q$ for $k=1,\ldots, d$. Then we get $v=x_{j_1}^{a_1} x_{j_2}^{a_2}\cdots x_{j_d}^{a_d}\in I$ with $v\notin Q$, which is a contradiction.

We construct $j_1<j_2<\cdots<j_t$ for $t=1,\ldots,d$, inductively. Since $\mu(Q)<i_1$, there exists $j_1\leq i_1$ such that $x_{j_1}\notin Q$. Assume now that the sequence $j_1, \ldots, j_t$ with $t<d$, has already been constructed, satisfying $j_1<j_2<\cdots<j_t$,   $j_r\leq i_r$  and $x_{j_r}\notin Q$ for $r=1,\ldots, t$. Since $i_{t+1}\geq i_1+t$, it follows that $A=\{1,\ldots,  i_{t+1}\}\setminus \{j_1,\ldots, j_t\}$ has at least $i_1$ elements. Therefore, since $\mu(Q)< i_1$, there exists $\l\in A$ such that $x_{\l}\notin Q$ and $l\leq i_{t+1}$. If $\l>j_t$, then we can choose $j_{t+1}=\l$. Otherwise, there exists a smallest integer $r\leq t$ such that $\l<j_r$. Then we rename the elements $j_1,\ldots,j_t$, and let $j_s'=j_s$ for $s<r$, $j_r'=\l$ and $j_{s+1}'=j_s$ for $s=r, \ldots ,t$. This new sequence of length $t+1$ satisfies all the requirements.  

Finally, we deal with the case that $m>1$. Let $h=\max\{\min(u_1),\ldots, \min(u_m)\}$. Then from what we have seen before, it follows that $I\subseteq (x_1,\ldots,x_h)$. This shows that $\height{(I)}\leq h$. Assume that $\height{(I)}< h$. Then there exists a monomial prime ideal $Q$ containing $I$ with $\mu(Q)<h$. Let $r$ be such that $\min(u_r)=h$. Since $B_k(u_r)\subset I\subseteq Q$, we  obtain a contradiction to the result we proved for $m=1$. 
\end{proof}

\section{The resolution of $k$-Borel ideals}

Let $I\subset S=K[x_1,\ldots,x_n]$ be a graded  ideal. The graded Betti numbers of $I$ can be computed by means
of the Koszul homology $H_i(\xb; S/I)$ with $ \xb=x_1,\ldots,x_n$. Indeed,  there is an isomorphism $H_i(\xb; S/I)\iso \Tor_i^S(K,S/I)$ of graded $K$-vector spaces. This isomorphism is even $\ZZ^n$-graded. In order to abbreviate notation, we set $H_i(j)= H_i(x_n,x_{n-1}\ldots,x_j; S/I)$. With this notation introduced we have $\beta_{i,j}(S/I)=\dim_K H_i(1)_j$. In order to compute the graded  Betti numbers of $S/I$, we determine a $K$-basis of $H_i(1)_\ab$ for each $i$  and $\ab\in \ZZ^n$. We will apply an inductive argument. For this reason, it is advisable to even determine a $K$-basis of the $K$-vector space $H_i(j)_\ab$ for all $i$, $j$ and $\ab\in \ZZ^n$. The homology class of a cycle $z$ in the Koszul complex $K_i(j)=K_i(x_n,\ldots,x_j;S/I)$ is an element of $H_i(j)$ and will be denoted  by $[z]_j$. When $j=1$, we simply write $[z]$.

The Koszul complex $K(\xb;S/I)$ is a complex of free $S/I$-modules. Let  $e_1,\ldots, e_n$ be the basis of $K_1(\xb;S/I)$ with $\partial (e_i)=x_i$ for all $i$. Then the elements $e_F$ with $F\subset [n]$ and $|F|=i$ form  the basis of $K_i(\xb;S/I)$. Here, for $F=\{j_1<j_2<\cdots <j_i\}$,  $e_F$ denotes  the wedge product $e_{j_1}\wedge \cdots \wedge e_{j_i}$ (cf. \cite[Appendix A.3]{HH}).  In $K_i(\xb;S/I)$ each basis element is annihilated by $I$. Thus for any $f\in S$,  $fe_F=\bar{f}e_F$, where  $\bar{f}=f+I$.
We call a vector $\ab\in \ZZ^n$, {\em $k$-bounded}, if $\ab =(a_1,\ldots,a_n)$ with $0\leq a_i\leq k$ for $i=1\ldots,n$. The multidegree of a monomial $u$ is denoted by $\Deg(u)$. 

\begin{Theorem}
\label{Koszulcycles}
Let $I\subset S=K[x_1,\ldots,x_n]$ be a $k$-Borel ideal,  and  let $\ab\in \ZZ^n$ be a $k$-bounded vector.  Then for $i>1$, $H_i(j)_\ab$ has a $K$-basis consisting of the homology classes of the cycles
\[
u'e_F\wedge e_{m(u)}, \quad u\in G(I) \quad\text{with}
\]
\[
F\subseteq [n], \quad |F| =i-1, \quad j\leq \min(F), \quad \max(F)< m(u) \quad\text{and} \quad \Deg(\xb^Fu)=\ab.
\]
Here, $m(u)=\max(u)$,  $u'= u/x_{m(u)}$, $\min(F)=\min\{i\:i\in F\}$ and $\max(F)=\max\{i\:i\in F\}$. Moreover, the $K$-basis of $H_1(j)_a$ consists of
the homology classes of the cycles $u'e_{m(u)}$, $u\in G(I)$ with $j\leq m(u)$ and $\Deg(u)=a$.
\end{Theorem}

\begin{proof}
We first notice if $u'e_F\wedge e_{m(u)}$ is a cycle in $K_i(j)_\ab$ for a monomial $u\in I$,  then  $u'e_F\wedge e_{m(u)}=0$, if $u\not\in G(I)$. In fact, suppose that $u\not\in G(I)$. Then $u=wv$ with $v\in G(I)$ and $w\neq 1$ a monomial. If $m(w)\geq m(v)$, then $u'=w'v\in I$ and  $u'e_F\wedge e_{m(u)}=0$. Now, let $m(w)<m(v)$. Then $u'=wv'$ and there exists $j<m(u)$ such that $x_j$ divides $w$, and hence $u'=(w/x_j)(x_jv')$. Since $x_jv'$ divides $u$ and since $u$ is $k$-bounded, it follows that $x_jv'$ is also $k$-bounded. Finally, since $I$ is a $k$-Borel ideal and $x_jv'\preceq v$, we conclude that $x_jv'\in I$, and hence $u'\in I$. Thus $u'e_F\wedge e_{m(u)}=0$, as desired. The conclusion is that for cycles of the form $u'e_F\wedge e_{m(u)}$ whose homology class is not zero, we have $u\in G(I)$.

We prove the theorem for $H_i(j)$ by backward induction starting with $j=n$.
We have $H_i(n)_\ab=0$ if $i>1$ and $H_1(n)=(I:x_n)e_n$. It follows that if  $H_1(n)_\ab\neq 0$,  then there exist a (unique)  $u\in G(I)$ with $m(u)=n$ and $\Deg(u)=\ab$, and then $[u'e_n]_n$ is the $K$-basis of $H_1(n)_\ab$.   This proves the assertion for $j=n$. Now let $j<n$, and assume that the theorem holds  for $j+1$.

We proceed by induction on $i$ to show that $H_i(j)_\ab$ has the desired $K$-basis. In order to prove this result  for $i=1$, we consider the exact sequence of Koszul  homology
\[
H_2(j)_\ab\to H_{1}(j+1)_{\ab-\epsilon_j}\to   H_1(j+1)_\ab\to H_1(j)_\ab\to H_0(j+1)_{\ab-\epsilon_j}\to H_0(j+1)_\ab.
\]

Here, $\epsilon_j$ is the $j$th standard unit-vector.  Let $U= \Ker(H_0(j+1)_{\ab-\epsilon_j}\to H_0(j+1)_\ab)$ and $V= \Im(H_1(j+1)_\ab\to H_1(j)_\ab$). Then a $K$-basis of $V$ together with a preimage in $H_1(j)_\ab$ of a $K$-basis  of $U$ establishes a $K$-basis of $H_1(j)_\ab$.

The map $H_0(j+1)\to H_0(j+1)$ is multiplication by $x_j$, and
$H_0(j+1)=S/(I,x_{j+1},\ldots,x_n)=S'/I'$, where $S'=K[x_1,\ldots, x_j]$ and $I'$ is a monomial ideal with $$G(I')=\{u\in G(I)\: \; m(u)\leq j\}.$$
 Thus, if $U\neq 0$, then $U$ has the $K$-basis consisting  of the residue class of $u'$, where $u\in G(I)$  with $m(u)=j$ and $\Deg(u)=\ab$. Its preimage in $H_1(j)_\ab$ is  $[u'e_j]_1$.  By assumption, $H_1(j+1)_\ab $ is generated by the elements $[u'e_{m(u)}]_{j+1}$  with $j+1\leq m(u)$ and $\Deg(u)=\ab$. Since $H_1(j+1)_\ab\to H_1(j)_\ab$ maps $[u'e_{m(u)}]_{j+1}$ to  $[u'e_{m(u)}]_j\in  H_1(j)_\ab$, the proof for $i=1$ is completed. Here we use that $H_1(j+1)_\ab\to H_1(j)_\ab$ is injective, because $H_2(j)_\ab\to H_{1}(j+1)_{\ab-\epsilon_j}$ is surjective, see the claim below. The injectivity of the mentioned map implies that $H_1(j+1)_\ab$ is isomorphic to its image $V$.

\medskip
Now let $i>1$, and consider the exact sequence
\[
 H_{i+1}(j)_\ab\to    H_{i}(j+1)_{\ab-\epsilon_j} \to   H_i(j+1)_\ab\to H_i(j)_\ab\to H_{i-1}(j+1)_{\ab-\epsilon_j}\to  H_{i-1}(j+1)_{\ab}.
\]
By induction hypothesis, if $i-1\geq 1$, then  $H_{i-1}(j+1)_{\ab-\epsilon_j}$ is generated by the homology classes $[u'e_F\wedge e_{m(u)}]_{j+1}$, where $u\in G(I)$ is a monomial,
$|F| =i-1$, $j+1\leq \min(F)$, $\max(F)< m(u)$ and $\Deg(\xb^Fu)=\ab-\epsilon_j$, together with the homology classes  $[u'\wedge e_{m(u)}]_{j+1}$ with  $j+1\leq m(u)$ and $\Deg(u)=\ab-\epsilon_j$, if $i-1=1$.

We claim that the map $H_i(j)_\ab\to H_{i-1}(j+1)_{\ab-\epsilon_j}$ is surjective for all $i-1\geq 1$. Indeed, given a generators
$[u'e_F\wedge e_{m(u)}]_{j+1}$ of   $H_{i-1}(j+1)_{\ab-\epsilon_j}$, we consider the element  $z=u'e_G\wedge e_{m(u)}$ with $G=F\union \{j\}$. Note that $F=\emptyset $, if $i-1=1$. Applying the Koszul differential  $\partial$  to $z$, we get $\partial(z) =\pm e_j\wedge \partial(u'e_F\wedge e_{m(u)})\pm x_ju'e_F\wedge e_{m(u)}$.  The first summand is zero, since $u'e_F\wedge e_{m(u)}$ is a cycle representing the homology class $[u'e_F\wedge e_{m(u)}]_{j+1}$. Since $\Deg(\xb^F u) =\ab-\epsilon_j$, and since $\ab$ is $k$-bounded, it follows that the exponent of $x_j$ in $u$ is $<k$. The inequality $j\leq m(u)$ and the fact that $I$ is $k$-Borel imply that $x_ju'\in I$. Here we use again the $k$-Borel property. Therefore, also the second summand  $x_ju'e_F\wedge e_{m(u)}$ is zero. We conclude that $z$ is a cycle, and hence $[z]_j=[u'e_G\wedge e_{m(u)}]_j\in H_i(j)_\ab$. Since $[u'e_G\wedge e_{m(u)}]_j$ is mapped to $\pm [u'e_F\wedge e_{m(u)}]_{j+1}$ via the map $H_i(j)_\ab\to H_{i-1}(j+1)_{\ab-\epsilon_j}$, we see  that $H_i(j)_\ab\to H_{i-1}(j+1)_{\ab-\epsilon_j}$ is surjective, as desired.

Therefore, the previous exact sequence splits for all $i-1\geq 1$ into the short exact sequence
\[
0 \to   H_i(j+1)_\ab\to H_i(j)_\ab\to H_{i-1}(j+1)_{\ab-\epsilon_j}\to 0.
\]
This implies that $H_i(j)_\ab$ has a $K$-basis consisting of  the images of the basis  elements of $ H_i(j+1)_\ab$ together with the preimages of the basis elements of  $H_{i-1}(j+1)_{\ab-\epsilon_j}$. By induction hypothesis,   $ H_i(j+1)_\ab$  has the  basis elements
 $[u'e_F\wedge e_{m(u)}]_{j+1}$ satisfying the conditions described in the theorem. The image of $[u'e_F\wedge e_{m(u)}]_{j+1}$ in $H_i(j)_\ab$
 is $[u'e_F\wedge e_{m(u)}]_{j}$. Also by induction, the basis elements   $H_{i-1}(j+1)_{\ab-\epsilon_j}$ are  known to be $[u'e_F\wedge e_{m(u)}]_{j+1}$ with the side conditions given in the theorem. A  preimage of $[u'e_F\wedge e_{m(u)}]_{j+1}$ in $H_i(j)_\ab$  is $[u'e_G\wedge e_{m(u)}]_j$ with $G=F\union \{j\}$.  From this we see that $H_i(j)_\ab$ has a $K$-basis   as described in the theorem.
\end{proof}

\begin{Corollary}
\label{j=1}
Let $I\subset S=K[x_1,\ldots,x_n]$ be a $k$-Borel ideal.  Then for $i>0$, a basis of $H_i(\xb;S/I)$
is given by the homology classes $u'e_F\wedge e_{m(u)}$  with
\[
u\in G(I), \quad  |F| =i-1, \quad \text{$\max(F)< m(u)$  and  $\Deg(\xb^Fu)$ is $k$-bounded.}
\]
\end{Corollary}

\begin{proof}
Since $H_i(\xb;S/I)\iso \Tor_i(K;S/I)$, it follows that $H_i(\xb;S/I)$ is a multigraded vector space whose graded components are $H_i(1)_\ab$. Since the exponent vectors of all monomials of $G(I)$ are $k$-bounded, the shifts in the multigraded free resolution of $S/I$ are also $k$-bounded.  This implies that $H_i(1)_\ab=0$ if $\ab$ is not $k$-bounded, see \cite[Theorem 3.1]{BH}. Hence
\[
H_i(\xb;S/I)=\Dirsum_\ab H_i(1)_\ab,
\]
where the direct sum is taken over all $k$-bounded integral vectors $\ab$. Therefore, the result follows from Theorem~\ref{Koszulcycles}.
\end{proof}

For $\ab\in \ZZ^n$ with $\ab=(a_1,\ldots,a_n)$, we set $|\ab|= \sum_{i=1}^na_i$. For a monomial $u=x_1^{a_1}\cdots x_n^{a_n}$, we set $\deg_{x_i}(u)=a_i$ for $1\leq i\leq n$. With this notation introduced we have

\begin{Corollary}
\label{somayeh}
Let $I\subset S=K[x_1,\ldots,x_n]$ be a $k$-Borel ideal. Then
\[
\beta_{i,i+j}(I)= \sum_{u\in G(I),\; \deg(u)=j} {m(u)-L(u)-1 \choose i},
\]
where $m(u)=\max(u)$ and $L(u)=|\{1 \leq \ell<m(u) \:\; \deg_{x_{\ell}}(u)=k\}|$.
\end{Corollary}

\begin{proof}
Let $\gamma_{u,i}=|\{F\subseteq [n]\:\; |F|=i-1, \; \max(F)<m(u), \; \Deg(\xb^Fu)\;\text{is $k$-bounded}\}|$.
By Corollary \ref{j=1}, we have
\[
\beta_{i,i+j-1}(S/I)= \sum_{u\in G(I),\; \deg(u)=j}\gamma_{u,i}.
\]
Thus
$\beta_{i,i+j}(I)=\beta_{i+1,i+j}(S/I)=\sum_{u\in G(I),\; \deg(u)=j}\gamma_{u,i+1}$. One can see that
$\gamma_{u,i+1}={m(u)-L(u)-1 \choose i}$.
\end{proof}

A monomial ideal $I\subset S$ with $G(I)=\{u_1,u_2,\ldots, u_m\}$ is said to have {\em linear quotients} with respect to the order $u_1,u_2,\ldots,u_m$ if for any $2\leq j\leq m$, the colon ideal  $(u_1,\ldots,u_{j-1}):(u_j)$ is generated by some variables. For any $j$, we define  $\set(u_j)=\{i: x_i\in (u_1,\ldots,u_{j-1}):(u_j)\}$. 

\medskip
We apply Corollary \ref{somayeh} to show the linear quotients property for equigenerated $k$-Borel ideals. For this we need 

\begin{Lemma}
\label{quotients}
Let $I$ be a monomial ideal generated in a single degree with $G(I)=\{u_1,\ldots,u_m\}$, and suppose that $I_j=(u_1,\ldots,u_j)$ has linear resolution for $j=1,\ldots,m$. Then $I$ has linear quotients for this order of the generators.
\end{Lemma}

\begin{proof}
Suppose that $I$ is generated in degree $d$. Then $I_{j+1}/I_j\iso (S/(I_j:u_{j+1}))(-d)$. Thus we get the short exact sequence
\[
0\to I_j\to I_{j+1}\to  (S/(I_j:u_{j+1}))(-d)\to 0,
\]
which induces the long exact sequence
\[
\Tor_1(K,I_{j+1})_{\ell}\to \Tor_1(K, S/(I_j:u_{j+1}))_{\ell-d}\to \Tor_{0}(K,I_j)_{\ell}\to  \Tor_{0}(K,I_{j+1})_{\ell}.
\]
For $\ell\neq d$, $\Tor_{0}(K,I_j)_{\ell}=0$, and for $\ell=d$,  $\Tor_{0}(K,I_j)_{\ell}\to  \Tor_{0}(K,I_{j+1})_{\ell}$ can be identified with the map $I_j/\mm I_j\to I_{j+1}/\mm I_{j+1}$. Since $I_j$ is generated by part of a minimal system of generators of $I_{j+1}$, this map is injective. Hence, for each  $\ell$, we get  the exact sequence $\Tor_1(K,I_{j+1})_{\ell}\to \Tor_1(K, S/(I_j:u_{j+1}))_{\ell-d}\to 0$. Since $I_{j+1}$ has $d$-linear resolution, it follows that $\Tor_1(K,I_{j+1})_{\ell}=0$ for $\ell\neq d+1$. This implies that $\Tor_1(K, S/(I_j:u_{j+1}))_{\ell}=0$ for $\ell\neq 1$, and shows that $(I_j:u_{j+1})$ is generated in degree 1, as desired.
\end{proof}

\begin{Corollary}
\label{anyorder}
Let $I$ be a $k$-Borel ideal generated in a single degree.
Choose any order $u_1,\ldots,u_m$ of the elements of $G(I)$ which extends the partial order $\prec_k$, i.e.\ if $u_i\prec_k u_j$, then $i<j$. Then $I$ has linear quotients with respect to this order of the generators.
\end{Corollary}

\begin{proof}
For the given order of the generators, it follows that for all $j$, the ideal $I_j=(u_1,\ldots,u_j)$ is $k$-Borel. Corollary~\ref{somayeh} implies that $I_j$ has linear resolution for all $j$.  Therefore, the result follows from   Lemma \ref{quotients}.
\end{proof}

In the following, it is shown that any $k$-Borel ideal (not necessarily generated in a single degree) has linear quotients with respect to the lexicographic order on its minimal generators. This order obviously is different from the order given in \Cref{anyorder}.

\begin{Proposition}\label{lq}
Let $I$ be a $k$-Borel ideal. Then $I$ has linear quotients. In particular, it is componentwise linear.
\end{Proposition}

\begin{proof}
Consider the lexicographic order  $u_1>\cdots>u_m$ of all elements in $G(I)$ which is induced by the order $x_1>\cdots>x_n$.
Let $1\leq i<j\leq m$, $u_i=x_1 ^{a_1}\cdots x_n ^{a_n}$, $u_j=x_1 ^{b_1}\cdots x_n ^{b_n}$ and $t$ be the smallest integer such that $x_t\mid (u_i:u_j)$. Since $u_i> u_j$, we have $a_r=b_r$ for any $r<t$ . Also $b_t\leq a_t-1\leq k-1$. There exists an integer $s>t$ such that $x_s\mid u_j$, otherwise $u_j\mid u_i$, which contradicts to $u_i,u_j\in G(I)$. Set $v=(u_j/x_s)x_t$. Then $v$ is a $k$-bounded  monomial since $b_t\leq k-1$. Therefore $v\in I$. Let $u_{\ell}\in G(I)$ be such that $u_{\ell}\mid v$. Let $u_{\ell}=x_1 ^{c_1}\cdots x_n ^{c_n}$. Then clearly $c_r\leq b_r$ for any $r\neq t$ and $c_t\leq b_t+1$.
If $c_t\leq b_t$, then $u_{\ell}\mid (u_j/x_s)$, which contradicts to $u_j\in G(I)$. Thus $c_t=b_t+1$.

We show that $c_r=b_r$ for any $r<t$. By contradiction assume that $c_r<b_r$ for some $r<t$. Then $w=(u_{\ell}/x_t)x_r\in I$ since it is $k$-bounded. Also by comparing exponents,  we get $(u_{\ell}/x_t)x_r\mid u_j$. Hence $(u_{\ell}/x_t)x_r=u_j$, which implies that $x_tu_j=x_ru_{\ell}$. Since $x_tu_j=x_sv$, we have $x_ru_{\ell}=x_sv$. This together with $u_{\ell}\mid v$ and $\deg(u_{\ell})=\deg(v)$ imply that $u_{\ell}=v$. It forces that $r=s$, which contradicts to $r<t<s$. So $c_r=b_r$ for any $r<t$, $u_{\ell}> u_j$ and $u_{\ell}:u_j=x_t$. This shows that $u_1,\ldots,u_m$ is an order of linear quotients for $I$.
\end{proof}

Let $I$ be a monomial ideal with linear quotients and $u_1,\ldots,u_m$ be an order of linear quotients for $I$. We denote by $M(I)$ the set of all monomials in $I$. The \emph{decomposition function of $I$} is defined as the map $g: M(I)\rightarrow G(I)$ given 
by  $g(u) = u_j$, where $j$ is the smallest number such that $u\in (u_1,\ldots,u_{j})$. 
The decomposition function of $I$ is said to be {\em regular} if for each $u\in G(I)$ and every
$s\in \set(u)$ we have
\[
 \quad\set(g(x_s u))\subseteq \set(u).
\]
The decomposition function of an ideal with linear quotients is not always regular. For example, consider $I = (x_2x_4, x_1x_2, x_1x_3)$. Then with respect to
the given order of the generators, $I$ has linear quotients, while $\set(x_1x_3) =
{2}$, and $\set(g(x_2(x_1x_3))) = {4}$. It is quite obvious that stable and squarefree stable ideals have
regular decomposition functions with respect to the reverse degree lexicographic
order. Another class of squarefree monomial ideals with regular decomposition function
is the Stanley-Reisner ideal of a matroid (see \cite[Theorem 1.10]{HT}). In the following proposition we show this property for any $k$-Borel ideal.

\begin{Proposition}\label{regdec}
Let $I$ be a $k$-Borel ideal. Then $I$ has a regular decomposition function.
\end{Proposition}

\begin{proof}
Consider the lexicographic order  $u_1>\cdots>u_m$ of all elements in $G(I)$ which is induced by the order $x_1>\cdots>x_n$. Let $g:M(I)\rightarrow G(I)$ be the decomposition function of $I$. We show that
$\set(g(x_s u_j))\subseteq \set(u_j)$ for any $u_j\in G(I)$ and any $s\in \set(u_j)$. Let $g(x_s u_j)=u_i$. Clearly $i\leq j$. If $i=j$, there is nothing to prove. So we may assume that $i<j$. There exists a monomial $w$ such that $x_su_j=u_iw$ and $\set(u_i)\cap \supp(w)=\emptyset$  by \cite[Lemma 1.7]{HT}. Let $u_j=x_1^{a_1}\cdots x_n^{a_n}$ and $u_i=x_1^{b_1}\cdots x_n^{b_n}$. Then $b_r\leq a_r$ for any $r\neq s$. Also note that  $x_s\nmid w$, otherwise $u_i\mid u_j$, a contradiction. So $b_s=a_{s}+1$. Since $u_i> u_j$,  $a_1=b_1,\ldots,a_{s-1}=b_{s-1}$.  In order to prove $\set(u_i)\subseteq \set(u_j)$, consider $t\in \set(u_i)$. Then by ~\Cref{lq}, we have $u_{\ell}:u_i=x_t$ for some $u_{\ell}> u_i$. Let $u_{\ell}=x_1^{c_1}\cdots x_n^{c_n}$, then $c_r\leq b_r$ for any $r\neq t$ and $c_t=b_t+1$.
Also since $u_{\ell}> u_i$, $b_1=c_1,\ldots,b_{t-1}=c_{t-1}$.
If $t=s$, then $t\in \set(u_j)$ and there is nothing to prove. So we may assume that $t\neq s$. We show that there exists an integer $q>t$ such that $a_q>0$. By contradiction assume that $a_r=0$ for any $r>t$.
We consider two cases and in each case we get a contradiction.

{\sc Case 1}.  Let $s<t$. Since $c_r\leq b_r\leq a_r$ for any $r>t$, we have  $c_r=b_r=0$ for any $r>t$. This implies that $u_{\ell}=x_t u_i$, which contradicts to $u_{\ell}\in G(I)$.

{\sc Case 2}.  Let $s>t$. Since $b_r\leq a_r$ for any $r>s$  and $s>t$, we have  $b_r=0$ for any $r>s$. This implies that $u_i=x_s u_j$, which contradicts to $u_i\in G(I)$.

Therefore there exists an integer $q>t$ such that $a_q>0$.
Since  $t\in \set(u_i)$ and $\set(u_i)\cap \supp(w)=\emptyset$, $x_t\notin \supp(w)$. Thus from the equality $x_su_j=u_iw$, we get $a_t=b_t$. Also the equality $c_t=b_t+1$ implies that $b_t<k$, since $u_{\ell}$ is a $k$-bounded monomial. Thus $a_t<k$.
Therefore $v=(u_j/x_q)x_t$ is a $k$-bounded monomial. So it belongs to $I$. Thus there exists some $p$ with $1\leq p\leq m$
such that $u_p\in G(I)$ and  $u_p\mid v$. With the same argument as in the proof of Proposition \ref{lq}, we get $u_p>u_j$ and  $u_p:u_j=x_t$. Hence $t\in \set(u_j)$.
\end{proof}

\medskip
In the case that $I$ has linear quotients, Herzog and Takayama \cite{HT} showed that the mapping cone
construction produces a minimal free resolution of $I$.
If furthermore $I$ has a regular decomposition function, then by \cite[Theorem 3.10]{DM} the minimal resolution of $I$ obtained as an iterated mapping cone is cellular and supported on a regular CW-complex. Hence as an immediate corollary of Proposition \ref{regdec} and \cite[Theorem 3.10]{DM} we have

\begin{Corollary}
Let $I$ be a $k$-Borel ideal. Then the minimal free resolution of $I$ obtained as an iterated mapping cone is cellular
and supported on a regular $CW$-complex.
\end{Corollary}

\section{Multiplicity and analytic spread of squarefree Borel ideals}

Let $I$ be a  squarefree monomial ideal. Then $I$ is called a {\em squarefree Borel ideal}, when $I$ is $1$-Borel. 
If $I$ is a squarefree Borel ideal, then the minimal set of  Borel generators of $I$ is the set of maximal elements of $G(I)$ with respect to $\preceq$. 
In this section we describe some algebraic invariants of squarefree Borel ideals.

Let $u= x_{i_1} x_{i_2}\cdots x_{i_d} $ be a squarefree monomial with $i_1<i_2<\cdots<i_d$. A {\em block} of $u$ is a subset $\{i_{\l}, i_{\l+1},\ldots,i_{\l+k}\}$ such that $i_{\l+j}=i_{\l}+j$ for $j=1,\ldots, k$.   A block of $u$  is called {\em maximal} if it is not properly contained in any other block of $u$. Note that $\supp(u)$  can be uniquely decomposed into the disjoint union of  maximal blocks of $u$. In other words, $\supp(u)=B_{1}\sqcup B_{2}\sqcup\cdots \sqcup B_{k}$, where each $B_i$ is a maximal block of $u$ and $\max\{j:\, j\in B_i\}<\min\{j:\, j\in B_{i+1}\}-1$ for all $i$. We call $B_i$ the $i$th block of $u$ and $B_{1}\sqcup B_{2}\sqcup\cdots \sqcup B_{k}$ the {\em block decomposition of $u$}.

For a set $A\subseteq [n]$, set $P_A=(x_i:\ i\in A)$. We use Proposition \ref{height} to prove

\begin{Theorem}
\label{multi}
Let $I$ be a squarefree principal Borel ideal with the Borel generator $u$ and let $B_1$ be the first block of $u$.  Then 
$P$ is a minimal prime ideal of $I$ of height $h=\height(I)$ if and only if $P=P_A$ for some $A\subseteq [1,\max(B_1)]$ with $|A|=h$. In particular,
\[
e(S/I)= {\max(B_1) \choose |B_1|-1}
.\]
\end{Theorem}
\begin{proof}
By Proposition \ref{height}, $\height(I)=\min(u)=h$. Hence $B_1=\{h,h+1,\ldots,h+k-1\}$, where $k=|B_1|$. 
First we show that for any $A\subseteq [1,h+k-1]$ of cardinality $h$, the ideal $P_A$ is a minimal prime ideal of $I$. Let $A$ be such a set and let $v\in G(I)$, that is, $v\preceq u$ and $\deg(v)=\deg(u)=d$. Let $v=x_{i_1}\cdots x_{i_d}$, where $i_1<i_2<\cdots< i_d$. Then $\{i_1,\ldots,i_k\}\subseteq [1,h+k-1]$. It follows that $A\cap \{i_1,\ldots,i_k\}\neq \emptyset$, otherwise $A\cup \{i_1,\ldots,i_k\} \subseteq [1,h+k-1]$ is a set of cardinality $h+k$, which is a contradiction. Thus 
$v\in P_A$ for any $v\in G(I)$, and hence $P_A$ is a minimal prime ideal of $I$. Clearly $\height(P_A)=h$.

Now, consider an arbitrary minimal prime ideal $P$ of $I$ of height $h$. Since $P$ is generated by variables, $P=P_A$ for some $A$ with $A\subseteq [1,\max(u)]$ and $|A|=h$. We show that $A\subseteq [1,h+k-1]$. Let $\supp(u)=B_{1}\sqcup B_{2}\sqcup\cdots \sqcup B_{r}$ be the block decomposition of $u$. If $r=1$, then $\supp(u)=B_1=\{h,h+1,\ldots,h+k-1\}$ and   $\max(u)=h+k-1$. Hence $A\subseteq [1,h+k-1]$ as desired.
Now suppose that $r\geq 2$ and by contradiction assume that there exists $j\in A$ with $j\geq h+k$. This implies that $|A\cap [1,h+k-1]|\leq h-1$. Let $d=\deg(u)$. Consider the monomial $v=\prod_{i=h}^{h+d-1} x_i$ of degree $d$. Clearly $v \prec u$ and   $v\in G(I)$. Set $J=B(v)$. Then $J\subset I\subseteq P_A$ with $\height(J)=\min(v)=h=\height(P_A)$. Hence $P_A$ is a minimal prime ideal of $J$ as well. Since $\supp(v)=[h,h+d-1]$ consists of one block, we have $A\subseteq [1,h+d-1]$. Let $x^A=\prod_{i\in A} x_i$. We show that for the monomial $w=(\prod_{i=1}^{h+d} x_i)/x^A$, we have $w\prec u$. Once we show this, since $\deg(w)=d$, we get $w\in G(I)$, while $w\notin P_A$, a contradiction.
Let $\supp(w)=\{j_1,\ldots,j_d\}$ and $\supp(u)=\{\l_1,\ldots,\l_d\}$ with $j_1<\cdots<j_d$ and  $\l_1<\cdots<\l_d$. Then $j_t\leq h+t\leq \l_t$ for any $k+1\leq t\leq d$. It remains to show that $j_t\leq \l_t=h+(t-1)$ for $1\leq t \leq k$. Since $|A\cap [1,h]|\leq h-1$, and $j_1$ is the smallest integer in $[1,h+d]\setminus A$, we have  $j_1\in [1,h]$. So $j_1\leq h=\l_1$. Similarly, we have $|A\cap ([1,h+1]\setminus \{j_1\})|\leq h-1$. Thus $j_2\in [1,h+1]$ and hence $j_2\leq h+1=\l_2$. The same argument shows the inequality $j_t\leq \l_t$ for any $1\leq t\leq k$, as desired. 

The second statement follows from the fact that the multiplicity of $S/I$ is equal to the number of the minimal prime ideals of $I$ of height $h$. 
\end{proof}

The formula for the multiplicity given in Theorem~\ref{multi} can be generalized to squarefree Borel ideals with any number of Borel generators under some extra assumption on the first block of the Borel generators, as the next result shows. One can easily construct many examples which show that this extra assumption  can not be dropped.  

\begin{Theorem}
\label{multigeneral}
Let $I$ be a squarefree Borel ideal with Borel generators   $u_{1},\ldots, u_{m}$ and let $B_{1i}$ be the first block of $u_i$ for $1\leq i\leq m$. Suppose there exists an integer $j$ such that $B_{1j}\subseteq B_{1i}$ for all $1\leq i\leq m$. Then  
\[
e(S/I)= {\max(B_{1j})\choose |B_{1j}|-1}
.\]
\end{Theorem}
\begin{proof}
Without loss of generality assume that $j=1$ and  $B_{11}\subseteq B_{1i}$ for all $i$. Let $B_{11}=[h,b]$. Then $h=\min(u_1)\geq \min(u_i)$ for any $1\leq i\leq m$. Set $J=B(u_1)$. Then by Proposition~\ref{height}, $\height(I)=\height(J)=\min(u_1)=h$. Therefore, it is enough to show that the set of minimal prime ideal of $I$ of height $h$ and that of $J$ are the same. If $P_A$ is a minimal prime ideal of $I$ with $\height(P_A)=h$, since $J\subseteq I\subseteq P_A$ 
and $\height(J)=h=\height(P_A)$, then $P_A$ is a  minimal prime ideal of $J$, as well. Now, let $P_A$ be a minimal prime ideal of $J$ of height $h$. Then by Theorem~\ref{multi}, $A\subseteq [1,b]$. For any $2\leq i\leq m$, let $J_i=B(u_i)$ and $B_{1i}=[a_i,b_i]$. Then by our assumption, $a_i\leq h\leq b\leq b_i$. Moreover, if $C_i$ be an arbitrary subset of $A$ with $|C_i|=a_i$, then $C_i\subseteq [1,b_i]$. By Theorem~\ref{multi}, $P_{C_i}$ is a minimal prime ideal of $J_i$. Hence $u_i\in J_i\subseteq P_{C_i}\subseteq P_A$ for all $i$. Hence $I\subseteq P_A$ and $P_A$ is a minimal prime ideal of $I$, as desired. 
\end{proof}	

In the next result we describe the analytic spread of an equigenerated squarefree  Borel ideal in terms of the block decomposition of its Borel generators. 
To prove this result, we use~\cite[Lemma 4.3]{DHQ} which relates the analytic spread of an equigenerated monomial ideal $I$ with linear relations to some combinatorial invariants of the so called linear relation graph of $I$.

Let $I$ be a monomial ideal with $G(I)=\{u_1,\ldots,u_m\}$. The {\em linear relation graph} $\Gamma$ of $I$ is the graph with edge set
\[
E(\Gamma)=\{\{i,j\}\: \text{there exist $u_k,u_l\in G(I)$ such that $x_iu_k=x_ju_l$}\}
\]
and vertex set $V(\Gamma)=\Union_{\{i,j\}\in E(\Gamma)}\{i,j\}$.

\begin{Theorem}
\label{analytic spread}
Let $I$ be an equigenerated squarefree  Borel ideal with Borel generators   $u_{1},\ldots, u_{m}$ which is not a principal ideal. For $i=1,\ldots,m$, let $B_{i1}\sqcup B_{i2}\sqcup\cdots \sqcup B_{ik_i}$ be the block decomposition of $u_i$ and set $n=\max\{\max(u_i)\;\:  1 \leq i\leq m\}$. Then
\[
\ell(I) =  \left\{
\begin{array}{ll}
n  &   \text{if $1\notin \bigcap_{i=1}^m B_{i1}$,}\\
n-|\bigcap_{i=1}^m B_{i1}|, & \text{otherwise.}\\
\end{array}
\right. \]
\end{Theorem}

\begin{proof}
Without loss of generality, let $u_m$ be a Borel generator with $n=\max(u_m)$.
We may assume that $\ell\notin\supp(u_m)$ for some $\ell<n$, since otherwise $u_m=\prod_{i=1}^n x_i$ and $I=(u_m)$, which is not the case.
For any $j\notin \supp(u_m)$, $v=x_j(u_m/x_n)\in G(I)$. Hence $x_ju_m=x_nv$, which implies that  $\{j,n\}\in E(\Gamma)$, where $\Gamma$ is the linear relation graph of $I$. 

First we consider the case that $1\notin \bigcap_{i=1}^m B_{i1}$. It is enough to show that $\Gamma$ is a connected graph on the vertex set $[n]$. Then, by ~\cite[Lemma 4.3]{DHQ}, we will get $\ell(I)=n$, as desired. By our assumption $1\notin B_{r1}$ for some $r$. First suppose that $r=m$ and $1\notin B_{m1}$.
Thus $1\notin \supp(u_m)$, it follows that $\{1,n\}\in E(\Gamma)$ and $v_j=x_1(u_m/x_j)\in G(I)$ for any $j\in\supp(u_m)$. So $\{1,j\}\in E(\Gamma)$ for any $j\in\supp(u_m)$. This implies that $\Gamma$ is a connected graph on $[n]$, as desired.

Now assume that $r\neq m$ and $1\in B_{m1}$. Let $B_{m1}=\{1,2,\ldots,s\}$. Then $s+1\notin\supp(u_m)$ and for any $j\in \supp(u_m)$ with $j>s$, we have $v=x_{s+1}(u_m/x_j)\in G(I)$. This  implies that $\{j,s+1\}\in E(\Gamma)$. 
In order to show that $\Gamma$ is connected, it is enough to show that for any $1\leq j\leq s$, $j$ is connected by a path to a vertex in $\{s+1,s+2,\ldots,n\}$.
Let $B_{r1}=\{t,t+1,\ldots,p\}$. If $s+1<t$, we have $\{j,s+1\}\in E(\Gamma)$ for any $j<s+1$. Indeed, one has $v=x_j(u_r/x_t)\in G(I)$ and $w=x_{s+1}(u_r/x_t)\in G(I)$  since  $j<s<t$ and  $B_{r1}=\{t,t+1,\ldots,p\}$.  It follows that $x_{s+1}v=x_jw$. So we may assume that $t\leq s+1$. 
 
 If $s+1\in\supp(u_r)$, since $t\leq s+1$,  we have $v=x_i(u_r/x_{s+1})\in G(I)$ for any $i<t$ or any $i$ with $t\leq i<s+1$ and $i\notin \supp(u_r)$. Then for such $i$'s we have $\{i,s+1\}\in E(\Gamma)$. If $t\leq i<s+1$ and $i\in\supp(u_r)$, then $\{1,i\}\in E(\Gamma)$ and $i,1,s+1$ is a path in $\Gamma$ and we are done in this case. 
 
 Now, we may assume that $s+1\notin\supp(u_r)$. Therefore $p<s+1\leq \deg(u_r)$ and hence there exists $h>s+1$ such that $h\in\supp(u_r)$. Then $v=x_{s+1}(u_r/x_h)\in G(I)$ and $w=x_{i}(u_r/x_h)\in G(I)$ for any $i<t$ or any $t\leq i\leq s$ with $i\notin\supp(u_r)$. Since $x_iv=x_{s+1}w$, we have
 $\{i,s+1\}\in E(\Gamma)$ for any $i<t$ or any $t\leq i\leq s$ with $i\notin\supp(u_r)$. Also for any $t\leq i\leq s$ with $i\in\supp(u_r)$, we have $\{1,i\}\in E(\Gamma)$ since $v=x_1(u_r/x_i)\in G(I)$. Hence $i,1,s+1$ is a path in $\Gamma$. So $\Gamma$ is connected.

Now, consider the case that $1\in \bigcap_{i=1}^m B_{i1}$ and let $\bigcap_{i=1}^m B_{i1}=[t]$. Then $z=\prod_{i=1}^t x_i$ divides any element in $G(I)$. Hence $I=zJ$, where $J$ is a monomial ideal in $K[x_{t+1},\ldots,x_n]$. The ideal $J$ is isomorphic to a squarefree Borel ideal $J'$ such that $x_1$ does not divide some Borel generator of $J'$. Since $J'$ is an ideal in the polynomial ring with $n-t$ variables and $J'\iso I$. By the first part of the proof, one has $\ell(I)=\ell(J)=n-t$.
\end{proof}

\section{Homological shift ideals of equigenerated squarefree Borel ideals}

In this section, for an equigenerated squarefree Borel ideal $I$, we obtain the Borel generators of the homological shift ideals of $I$. Moreover, we study the behaviour of some algebraic invariants of these shift ideals.

Let
 $I\subset S$ be a monomial ideal with minimal multigraded free $S$-resolution
\[
\FF: 0\longrightarrow F_q\longrightarrow F_{q-1}\longrightarrow\cdots \longrightarrow F_1\longrightarrow F_0\longrightarrow I\longrightarrow 0,
\]
where $F_i=\Dirsum_{j=1}^{b_i}S(-\ab_{ij})$. The vectors $\ab_{ij}$ are called the {\em multigraded shifts} of the resolution  $\FF$.
The monomial ideal $\HS_i(I)=(\xb^{\ab_{ij}}\: j=1,\ldots,b_i)$ is called the {\em $i$th homological shift ideal} of $I$. Note that $\HS_0(I)=I$. 

Let $u$ be a squarefree monomial. A positive integer $i$ is called a {\em gap} of $u$ if $i<\max(u)$ and  $x_i$ does not divide $u$.  The set of all gaps of $u$ is denoted by $\gap(u)$ and the maximal element of $\gap(u)$ is called the {\em maximal gap} of $u$. 

\begin{Proposition}
\label{gen1}
Let $I=B_1(u_{1},\ldots, u_{m})$ be an equigenerated squarefree Borel ideal. Then  
\[
\HS_1(I)=B_1(x_{p_1}u_1,\ldots,x_{p_m}u_m),
\]
where for each $i$, $p_i$ is the maximal gap of $u_i$.
\end{Proposition}

\begin{proof}
By \cite[Proposition 3.1]{BJT}, $\HS_1(I)$ is a squarefree Borel ideal and $x_{p_i}u_i\in \HS_1(I)$. So 
\[
B_1(x_{p_1}u_1,\ldots,x_{p_m}u_m)\subseteq \HS_1(I).
\]
Now, consider a generating monomial of $\HS_1(I)$, which by Corollary~\ref{j=1} is of the form $x_{\ell}v$  for some $v\in G(I)$ and some $\ell\in\gap(v)$.  Then by~\Cref{partial}, there exists $1 \leq t\leq m$ such that $v\preceq u_t$. We show that $x_{\ell}v\preceq x_{p_t}u_t$. Without loss of generality we may assume that $\ell$ is the maximal gap of $v$.  Let $v=x_{j_1}\cdots x_{j_{d}}$, $u_t=x_{i_1}\cdots x_{i_{d}}$ with  $j_1<\cdots< j_{d}$ and $i_1<\cdots<i_{d}$, where $d$ is the degree of the minimal monomial generators of $I$. We have $j_{r}\leq i_{r}$ for all $1 \leq r\leq d$. We set $j_0=i_0=0$.
Let $s$ and $k$ be integers with $0 \leq s\leq d-1$ and $0 \leq k\leq d-1$ such that $j_s<\ell<j_{s+1}$ and $i_k<p_t<i_{k+1}$. Then $\ell=j_{s+1}-1$ and $p_t=i_{k+1}-1$. 
We may write $x_{\ell}v=x_{j'_1}\cdots x_{j'_{d+1}}$, $x_{p_t}u_t=x_{i'_1}\cdots x_{i'_{d+1}}$, where $j'_1<\cdots<j'_{d+1}$ and $i'_1<\cdots<i'_{d+1}$. Then $j'_{s+1}=\ell=j_{s+1}-1$ and $i'_{k+1}=p_t=i_{k+1}-1$. 

First suppose that $k=s$. Then $j'_r=j_r\leq i_r=i'_r$ for any $r<k+1$. Also  $j'_{k+1}=\ell=j_{k+1}-1\leq i_{k+1}-1=i'_{k+1}=p_t$ and $j'_r=j_{r-1}\leq i_{r-1}=i'_r$ for any $r>k+1$. Hence $x_{\ell}v\preceq x_{p_t}u_t$, as desired. 

Now, suppose that $s<k$. Then $j'_r=j_r\leq i_r=i'_r$ for any $r\leq s$. Also $j'_{s+1}=\ell=j_{s+1}-1\leq i_{s+1}-1=i'_{s+1}-1<i'_{s+1}$ and for any $s+2\leq r\leq d+1$, we have $j'_r=j_{r-1}\leq i_{r-1}$. Note that $i_{r-1}=i'_{r-1}$ or $i_{r-1}=i'_r$, which implies that $i_{r-1}\leq i'_r$. Hence $j'_r\leq i'_r$ for all $r$, as desired.

Finally, consider the case that $s>k$. Then $j'_r=j_r\leq i_r=i'_r$ for any $1\leq r\leq k$. Also $j'_r=j_d-(d-r+1)\leq i_d-(d-r+1)=i'_r$ for any $s+1\leq r\leq d+1$. 
Now, suppose that
$k+1\leq r\leq s$. Then $j'_{r}=j_r<j_d-(d-r)\leq i_d-(d-r)$. Since $i_{k+1}-1$ is the maximal gap of $u_t$, we have $i_d-(d-r)=i_r$ for any $k+1\leq r\leq s$. Hence   $j'_r\leq i_r-1$ for any $k+1\leq r\leq s$. When $k+2\leq r\leq s$, we have $i_r-1=i_{r-1}=i'_r$ and hence $j'_r\leq i'_r$. Also when $r=k+1$, we have $j'_{k+1}\leq  i_{k+1}-1=i'_{k+1}$.  Therefore, $x_{\ell}v\preceq x_{p_t}u_t$.
\end{proof}

As an immediate corollary of~\Cref{gen1}, we have 

\begin{Corollary}\label{clever}
Let $I=B_1(u_{1},\ldots, u_{m})$ be an equigenerated squarefree Borel ideal. Then for any $k\geq 1$,
\begin{itemize}
    \item[(a)] $\HS_k(I)=B_1(x_{p_{11}}\cdots x_{p_{1k}}u_1,\ldots,x_{p_{m1}}\cdots x_{p_{mk}}u_m)$, where  ~$p_{i1},\ldots,p_{ik}$ are maximal possible distinct integers in $\gap(u_i)$. 
    \item[(b)] $\HS_1(\HS_k(I))=\HS_{k+1}(I)$.
\end{itemize}
\end{Corollary}

Using the description of the Borel generators of homological shift ideals $\HS_k(I)$ in Corollary~\ref{clever}, the description for the height in Proposition~\ref{height} and analytic spread in Theorem \ref{analytic spread} we get

\begin{Corollary}
\label{nond}
Let $I=B_1(u_{1},\ldots, u_{m})$ be an equigenerated squarefree Borel ideal. Then $\height(\HS_{k+1}(I))\leq \height(\HS_{k}(I))$ and $\ell(\HS_{k+1}(I))\leq \ell(\HS_{k}(I))$ for all $k$.
\end{Corollary}

A similar result to ~\Cref{nond} does not hold for the multiplicity of the homological shift ideal of equigenerated squarefree Borel ideals. However, we have 

\begin{Proposition}\label{easy but lengthy}
Let $I=B_1(u)$ be a squarefree principal Borel ideal. Then $e(S/\HS_{k}(I))$ is a unimodal function of $k$. 
\end{Proposition}
\begin{proof}
Let $\supp(u)=B_{1}\sqcup B_{2}\sqcup\cdots \sqcup B_{k}$ be the block decomposition of $u$ with $B_1=[a,b]$. Then by Corollary~\ref{clever}, $\HS_{k}(I)=B_1(u_k)$ for a monomial $u_k$ for all $k$. Let $[a_k,b_k]$ be the first block in the block decomposition of $u_k$. Then there exists $k_0$ such that $b_k=b$ and $a_k=a$ for $1\leq  k\leq k_0$ and for $k>k_0$ we have $b_k=n=\max(u)$ and $a_k=a-(k-k_0-1)$. Therefore, by Theorem \ref{multi}, one has  $e(S/\HS_{k}(I))=e(S/B_1(u_k))=\binom{b}{a}=e(S/I)$ for $1\leq  k\leq k_0$ 
and  $e(S/\HS_{k}(I))=\binom{n}{a-(k-k_0-1)}$ for $k>k_0$. This proves the assertion.
\end{proof}

\section{Homological shift ideals of $t$-spread Veronese  ideals}

In \cite{EHQ} the concept of a $t$-spread monomial was introduced. A monomial $x_{i_1}x_{i_2}\cdots x_{i_d}$ with $i_1\leq i_2\leq \dots \leq i_d$ is called {\it $t$-spread} if $i_j -i_{j-1}\geq t$ for $2\leq j\leq d$. We fix integers $d$ and $t$. The monomial ideal in $S=K[x_1,\ldots,x_n]$ generated by all $t$-spread monomials of degree $d$ is called the {\it t-spread Veronese ideal of degree $d$}. We denote this ideal by $I_{n,d,t}$. For $t=1$, one obtains the squarefree Veronese ideals, which may also be viewed as the edge ideals of  hypersimplexes. Properties of these ideals were first studied  in \cite{St}. The $K$-subalgebra of $S$ generated by the monomials $v\in G(I_{n,d,t})$ is called the {\it $t$-spread Veronese algebra}. In  \cite{Di} the Gorenstein property for the $t$-spread Veronese algebras was analyzed.

Let $t\geq 1$ be an integer, $u=x_{i_1}x_{i_2}\cdots x_{i_d}$ be a $t$-spreal monomial with $i_1< i_2< \cdots < i_d$. A {\em $t$-block} of $u$ of size $r$ is a subset $B\subseteq \supp(u)$ such that
$B=\{i_k,i_{k+1},\ldots,i_{k+r-1}\}$ with $i_{l+1}-i_l=t$ for all $k\leq l\leq k+r-2$. A $t$-block of $u$ is called {\em maximal} if it is not contained in any other $t$-block of $u$. The set $\supp(u)$ can be uniquely decomposed into the disjoint union of maximal $t$-blocks of  $u$, say $\supp(u)=B_{1}\sqcup B_{2}\sqcup\cdots \sqcup B_{k}$, where each $B_i$ is a maximal block and $\max\{j:\, j\in B_i\}<\min\{j:\, j\in B_{i+1}\}-t$ for all $i$. 
%
%
For $j=0,\ldots,r-1$,  the  {\em $j$th gap  interval} of $u$ is the set
$L_j=[ \max(B_j)+t,\min(B_{j+1})-1]$, where $B_0=\{-t+1\}$.
The union of all gap intervals of $u$ is denoted by $\gap(u)$ and any element of $\gap(u)$ is called a {\em gap} of $u$.

Let $u=x_{i_1}x_{i_2}\cdots x_{i_d}$ be a monomial with $i_1\leq i_2\leq \cdots\leq i_d$ and let $t$ be a positive integer. A pair $(i_k,i_{k+1})$ with $1\leq k\leq d-1$ is called a {\em $t$-irregular} pair of $u$ if $i_{k+1}-i_k<t$.  
\begin{Theorem}
\label{gent}
Let $n,d$ and $t$ be positive integers with $d,t\leq n$, and $t\geq 1$ and let $I=I_{n,d,t}$. Then 
\[
\HS_k(I)=(x_{i_1}x_{i_2}\cdots x_{i_{d+k}}:\ i_1<\cdots<i_{d+k},\ \text{and}\ \  i_{\ell+1}-i_{\ell}< t \ \text{for at most}\ k\  integers).
\]
\end{Theorem}

\begin{proof}
By \cite[Lemma 1.5]{HT}, one has
$$\HS_k(I)=(ux_{i_1}x_{i_2}\cdots x_{i_k}:\ u\in G(I),\  i_1<\cdots<i_k, \ \{i_1,\ldots,i_k\}\subseteq \set(u)).$$ 
Moreover, $I$ has linear quotients by \cite[Theorem 2.1]{DHQ}. For any $u\in G(I)$, we have $i\in \set(u)$ if and only if $i$ belongs to some gap interval of $u$ by \cite[Lemma 2.2]{DHQ}. Consider $u\in \G(I)$ and $\{i_1,\ldots,i_k\}\subseteq \set(u)$. Let $L_1,\ldots, L_m$ be the gap intervals of $u$, and let $k_j$ be the number of elements in $\{i_1,\ldots,i_k\}$ which belong to the gap interval $L_j$. Then $k=\sum_{j=1}^m k_j$.  Let $L_j=[\max(B_j)+t,\min(B_{j+1})-1]$ and $A_j=\{i_1,\ldots,i_k\}\cap L_j$. Then, for any $1\leq j\leq m$, the monomial $u\prod_{i\in A_j} x_{i}$ has at most $k_j$ $t$-irregular pairs  because of $\min(A_j)-\max(B_j)\geq t$. Moreover, for distinct integers $j$ and $s$ with $j<s$, we have $\min(A_s)-\max(A_j)\geq \min(L_s)-\max(L_j)\geq t$. This implies that $ux_{i_1}x_{i_2}\cdots x_{i_k}$ has at most $\sum_{j=1}^m k_j=k$ $t$-irregular pairs.

Conversely, assume that $v=x_{i_1}\cdots x_{i_{d+k}}$ is a monomial with at most $k$ $t$-irregular pairs and let $i_1<\cdots<i_{d+k}$. Let $\{j_1,\ldots,j_k\}\subseteq \supp(v)$ which contains the smallest element from each $t$-irregular pair of $v$. Then $u=v/\prod_{\ell=1}^k x_{j_\ell}$ is a minimal generator of $I$ and hence $v\in \HS_k(I)$.   
\end{proof}

It is conjectured that all the homological shift ideals of $t$-spread Veronese ideals have linear quotients. In the next result, we provide a proof only for the first shift ideal, which is still rather complicated. The expected order of linear quotients for higher shift ideals is similar to the one used in the following proof. 

\begin{Theorem}
\label{hs1}
Let $n,d$ and $t$ be positive integers with $d,t\leq n$ and $t\geq 1$ and let $I=I_{n,d,t}$. Then $\HS_1(I)$ has linear quotients. 
\end{Theorem}

\begin{proof}
We set $J=\HS_1(I)$. For any  monomial  $w\in G(J)$, we have $w=x_iu$ for a  monomial $u\in G(I)$ and some $i\in \gap(u)$. Such a presentation is called to be the right presentation of $w$ if $u$ is the largest monomial with respect to the lexicographic order (induced by $x_1>x_2>\cdots >x_n$) among all the possible monomials $v\in G(I)$ that appear in some presentation of $w$.  
We consider a total order on  $G(J)$  as follows:
Let $x_iu, x_jv\in G(J)$ be such that $x_iu\neq x_jv$, both $x_iu$ and $x_jv$ are the  right presentations. We set $x_iu>x_jv$ if $u>_{lex} v$ or $u=v$ and $i<j$.
We show that $J$ has linear quotients with respect to this order.

If $u=v$, then $x_iu:x_jv=x_iu:x_ju=x_i$ and there is nothing to prove. Let $u>_{\lex} v$. Let $u=\prod_{m=1}^d x_{r_m}$ with $r_1< r_2<\cdots< r_d$  and $v=\prod_{m=1}^d x_{s_m}$ with $s_1< s_2< \cdots < s_d$. Since $u>_{\lex} v$, there exists an integer $1\leq \ell \leq d$ such that 
$r_1=s_1$, $\dots$, $r_{\ell-1}=s_{\ell-1}$ and $r_{\ell}<s_{\ell}$.
By \cite[Theorem 2.1]{DHQ}, $I$ has linear quotients with respect to the lexicographic order. Moreover,  $w=(v/x_{s_{\ell}})x_{r_{\ell}}\in G(I)$ and $w>_{\lex} v$. Put $k=r_{\ell}$, one has 
$w:v=x_k$ and $x_k|(u:v)$. If $\ell=d$, then $w=u$. We will show that $x_ju\in J$. Indeed, if $j\in \gap(u)$, then the assertion is clear. If $j\notin \gap(u)$, then $j>k$. Moreover, $x_ju$ can be written as $x_ju=x_ku'$, where $u'=(x_ju/x_k)$ and $k\in \gap(u')$. Hence $x_ju=x_ku'\in J$ and $x_ju:x_jv=x_k$, as desired. So we may assume that $\ell<d$.  

\medskip
{\sc Case 1}. First assume that $k\neq j$. If $j\in \gap(w)$, then $x_jw\in J$. Moreover, $x_jw>x_jv$ since $w>_{\lex} v$. Since  $w=(v/x_{s_{\ell}})x_{k}$, we have 
$x_jw:x_jv=x_k$ and $x_k|(x_iu:x_jv)$. Hence we are done. So we may assume that $j\notin \gap(w)$. Since $x_j\in \gap(v)$, we have either $j<s_1$ or $s_{p-1}<j<s_p$ for some $p>1$ with $j\geq s_{p-1}+t$. 
If $j<s_1$,  then $\ell=1$ and $k<j$ because of  $j\notin \gap(w)$. Therefore $x_jw$ can be written as $x_jw=x_kw'$, where $w'=x_jw/x_k$. This means $w'\in G(I)$, because $w'=x_jx_{s_2}\cdots x_{s_d}$ and $s_2-j\geq s_2-s_1\geq t$. Moreover, $k\in \gap(w')$ since $k<j$. So $x_kw'\in J$ and since $w>_{\lex} v$, we have $x_kw'>x_jv$. Also $x_kw':x_jv=x_jw:x_jv=x_k$. 
Therefore we are done in the case that $j<s_1$. Now assume that  $s_{p-1}<j<s_p$ for some $p>1$, with $j\geq s_{p-1}+t$. 
If $p\neq \ell$, then $j\in \gap(w)$ which contradicts to  $j\notin \gap(w)$. Hence
 $p=\ell$, which implies $s_{\ell-1}<j<s_{\ell}$ and $j\geq s_{\ell-1}+t$. Note that $s_{\ell-1}<k<s_{\ell}$. If $j<k$, then $j$ belongs to the gap interval $[s_{\ell-1}+t,k]$ of $w$, which contradicts to $j\notin \gap(w)$. So we have $j>k$. Also since $[k+t,s_{\ell+1}]$ is a gap interval of $w$, $j$ does not belong to this interval. So $j<k+t$. 
Then for $w'=(x_jw)/x_k$, we have $w'\in G(I)$ and $k\in \gap(w')$. So $x_jw=x_kw'\in J$ and $x_kw'>x_jv$, since $w>_{\lex} v$. Moreover, $x_jw:x_jv=x_k$, as desired.
\medskip

{\sc Case 2}. Let $k=j$. Then $s_{\ell-1}<j<s_{\ell}$. First we show that  $s_{\ell}< j+t$. Suppose in contrary that $s_{\ell}\geq j+t$. Since $\ell<d$, we have $s_{\ell}\in \gap(w)$. Since  $x_jv=x_{s_{\ell}}w$ and $w>_{\lex} v$, $x_jv$ is not a right presentation, a contradiction.  Hence  $s_{\ell}< j+t$ and $s_{\ell}\notin \gap(w)$. 
If there exists $q\in \supp(x_iu)\cap \gap(w)$, then $x_qw\in J$, $x_qw:x_jv=x_q$ and $x_q|(x_iu:x_jv)$ because $s_{\ell}\notin \gap(w)$ and $q\in \gap(w)$. So in this situation we are done. Now, assume that $\supp(x_iu)\cap \gap(w)=\emptyset$ which in particular implies that $i\notin \gap(w)$. Hence for the gap interval $L=[j+t,s_{\ell+1}-1]$ of $w$, we have $L\cap \supp(u)=\emptyset$. This implies that $r_{\ell+1}\geq s_{\ell+1}$. So $r_{\ell+2}\geq r_{\ell+1}+t\geq s_{\ell+1}+t$. Now, consider the gap interval $L'=[s_{\ell+1}+t,s_{\ell+2}-1]$ of $w$. Since $L'\cap\supp(u)=\emptyset$, by the inequality $r_{\ell+2}\geq s_{\ell+1}+t$, we have $r_{\ell+2}\geq s_{\ell+2}$. By the similar arguments, we have $r_p\geq s_p$ for any $\ell+1\leq p\leq d$. Since $i\in\gap(u)\setminus \gap(w)$, we have $u\neq w$.  So there exists $p$ such that $r_p>s_p$. Let $b=\max\{p:\ r_p>s_p\}$. Then $r_{p+1}=s_{p+1}$ and $r_p>s_p$. Then $x_{r_p}w$  can be written as $x_{r_p}w=x_{s_p}w'$, where $w'=(x_{r_p}w/x_{s_p})$. Since $r_p-s_{p-1}>s_p-s_{p-1}\geq t$, and $s_{p+1}-r_p=r_{p+1}-r_p\geq t$, we conclude that $w'$ is a $t$-spread monomial. It forces that $w'\in G(I)$. Also clearly $s_p\in \gap(w')$. Hence $x_{r_p}w=x_{s_p}w'\in J$ and since $w'>_{\lex} v$, we have $x_{r_p}w>x_jv$. Note that $x_{r_p}w:x_jv=x_{r_p}$, as desired.
\end{proof}

\begin{center}
{\bf Acknowledgement}
\end{center}
The present paper was in large parts completed while the authors stayed at Mathematisches
Forschungsinstitut in Oberwolfach, September 5 to September 25, 2021, in the frame of
the Research in Pairs Program.  The second author was supported by a 
grant awarded by CIMPA.
The fourth author is supported by the National Natural Science Foundation of China (No.~11271275) and by foundation of the Priority Academic Program Development of Jiangsu Higher Education Institutions.

\medskip

\end{document}